\documentclass[reqno,a4paper,11pt]{amsart}

\usepackage[all,poly]{xy}
\usepackage{amsfonts}
\usepackage[mathcal]{eucal}
\usepackage{eufrak}
\usepackage{amssymb}
\usepackage{amsmath}
\usepackage{mathrsfs}
\usepackage{color}
\usepackage[colorlinks]{hyperref}
\usepackage{enumerate}

\theoremstyle{plain}
\newtheorem {lemma}{Lemma} 
\newtheorem {theorem}[lemma]{Theorem}

\newtheorem {proposition}[lemma]{Proposition}

\theoremstyle{definition}
\newtheorem {remark}[lemma]{Remark}

\theoremstyle{definition}

\newtheorem{deff}[lemma]{Definition}{}

\newcommand{\M}{\operatorname{\mathbb M}}
\newcommand{\LL}{\operatorname{\mathcal L}}

\newcommand{\gr}{\operatorname{gr}}

\newcommand\Gr[1][]{{\operatorname{{Gr}^{#1}-}}}

\newcommand{\Ga}{\Gamma}
\newcommand{\ga}{\gamma}

\newcommand{\de}{\delta}
\newcommand{\la}{\lambda}

\newcommand{\Modd}{\operatorname{Mod-}}

\title{Leavitt path algebras are graded von Neumann regular rings}

\author{Roozbeh Hazrat}\address{
Centre for Research in Mathematics\\
University of Western Sydney\\
Australia}

\email{r.hazrat@uws.edu.au}

\subjclass[2000]{16D70} 
\keywords{Leavitt path algebras, von Neumann regular rings, graded von Neumann regular rings}

\begin{document}
\begin{abstract}
In sharp contrast to the Abrams-Rangaswamy Theorem that the only von Neumann regular Leavitt path algebras are exactly those associated to acyclic graphs, here we prove that the Leavitt path algebra of any arbitrary graph is a {\it graded} von Neumann regular ring. Several properties of Leavitt path algebras, such as triviality of the Jacobson radical, flatness of graded modules and finitely generated graded right (left) ideals being generated by an idempotent element,  follow as a consequence of general theory of grade von Neumann regular rings. 
 
\end{abstract}

\maketitle

 \setcounter{section}{1}

The von Neumann regular rings constitute an important class of rings. A unital ring $A$ is von Neumann regular (or regular for short), if for any $x \in A$, we have $x \in xAx$. 
There are several equivalent module theoretical definitions, such as $A$ is regular if and only if any module over $A$ is flat. Ken Goodearl's book~\cite{goodearlbook} is devoted to this class of rings. 
The definition extends to non-unital ring in an obvious manner.  

If a ring has a graded structure, one defines the graded version of regularity in a natural way: the graded ring $A$ is graded von Neumann regular (or graded regular for short), if for any homogeneous element $x\in A$, we have $x\in xAx$. Many of the module theoretic properties established for von Neumann regular rings can be extended to the graded setting; For example, $A$ is graded regular if and only if any graded module is flat.

To any arbitrary direct graph, one can associate a so called Leavitt path algebra~\cite{aap05,amp}. Despite being introduced less than a decade ago,  Leavitt path algebras have arisen in a variety of different contexts as diverse as analysis, symbolic dynamics, noncommutative geometry, representation theory, and number theory. One line of interests in these algebras is to find classes of graphs which give specific types of Leavitt path algebras.
 
In~\cite{abramsranga}, Abrams and Rangaswamy have shown that the Leavitt path algebra of an arbitrary graph is von Neumann regular if and only if the graph is acyclic (i.e., contains no cycle). 
Since Leavitt path algebras have canonical $\mathbb Z$-graded grading, it is natural to ask when these rings are graded von Neumann ring. In contrast with the Abrams-Rangaswamy result which restricts regular Leavitt path algebras to only acyclic graphs, here we prove that  (Theorem~\ref{sthfin3}) the Leavitt path algebra associated to any arbitrary graph is graded von Neumann regular.  Several properties of Leavitt path algebras, such as triviality of the Jacobson radical, flatness of graded modules and finitely generated graded right (left) ideals being generated by an idempotent element,  follow as a consequence of general theory of grade von Neumann regular rings. 

Let $A$ be a strongly graded ring. By Dade's theorem, $\Gr A$ is equivalent to $\Modd A_0$; Here $\Gr A$ is the category of graded right $A$-modules and $\Modd A_0$ is the category of right modules over the ring of zero component $A_0$  (see~\cite[\S1.5]{hazgrmon}). Since any (graded) flat module is a direct limit of (graded) projective modules, from the equivalence of categories above, it follows that  $A$ is graded regular if and only if $A_0$ is regular. On the other hand, for a finite graph with no sinks, its Leavitt path algebra is strongly graded (\cite[Theorem~1.6.9]{hazgrmon}). We also know that its zero component ring is an ultramatricial algebra which is regular (see the proof of~\cite[Theorem~5.3]{amp}). Putting these together we get that Leavitt path algebras of finite graphs with no sinks are graded regular ring. However in order to show that all Leavitt path algebras are graded regular we need a different approach.

 The idea of the proof (Theorem~\ref{sthfin3}) is as follows. We start off with finite graphs. We show that removing sources from a graph does not change the associated Leavitt path algebras up to graded Morita equivalent. Since the graded regularness is graded Morita invariant (Proposition~\ref{cafejen4may}), the proof reduces to  finite graphs with no sources. The Leavitt path algebras of finite graphs with no sources can be realised as corner skew Laurent polynomial rings (\S\ref{cornerskew}). For these rings Proposition~\ref{lanhc8} gives the condition when they are graded regular which is the case for Leavitt path algebras.  This proves the theorem for all finite graphs. Since a direct limit of graded regular rings is graded regular, the theorem extends to row finite graphs, as they can be written as a direct union of finite graphs. Assigning a suitable (non-canonical) grading to the Leavitt path algebra of a countable graph, we see that there is 
  a graded monomorphism from this algebra to its desingularization which is a row finite graph (\S\ref{hngtrere66}). This will let us extend the theorem to countable graphs. Invoking Goodearl's treatment of Leavitt path algebras of arbitrary graphs as a direct limit of countable graphs~\cite{goodearllpa}, the theorem will be extended to all graphs.

\subsection{Graded von Neumann regular rings}

Recall that a (not necessarily unital)  ring $A$ is called a
\emph{$\Ga$-graded ring}, or simply a \emph{graded ring},
if $ A= \textstyle{\bigoplus_{ \ga \in \Ga}} A_{\ga}$, where 
\index{graded ring} $\Ga$ is an (abelian) group, each $A_{\ga}$ is
an additive subgroup of $A$ and $A_{\ga}  A_{\delta} \subseteq
A_{\ga + \delta}$ for all $\ga, \delta \in \Ga$. 

The set $A^{h} =
\bigcup_{\ga \in \Ga} A_{\ga}$ is called the set of {\it homogeneous elements} of
$A$. The non-zero elements of $A_\ga$ are called \emph{homogeneous of degree $\ga$}
and we write deg$(a) = \ga$ if $a \in A_{\ga}\backslash \{0\}$. We call the set \[\Gamma_A=\{ \ga \in \Gamma \mid A_\ga \not = 0 \}\] the {\it support} of $A$. We say $A$ has a \emph{trivial grading}, or $A$ is \emph{concentrated in
degree zero} if the support of $A$ is the trivial group, i.e., $A_0=A$ and $A_\ga=0$ for $\ga \in \Gamma \backslash \{0\}$. For a $\Gamma$-graded unital ring $A$ (with identity element $1$), one can prove that  $1$ is a homogeneous  element of degree $0$,  $A_0$ is a subring of $A$ and for an invertible element $a \in A_{\ga}$, its inverse
$a^{-1}$ is homogeneous of degree $-\ga$, i.e., $a^{-1} \in A_{-\ga}$. 

A $\Gamma$-graded (not necessarily unital)    ring $A$ is called a {\it graded von Neumann regular} ring, if for any homogeneous element $x\in A$, there is $y\in A$ such that $xyx=x$. Note that $y$ can be chosen to be a homogeneous element. Throughout the note, we call such rings also {\it graded regular rings}. We refer the reader to 
\cite[C, I.5]{grrings} for a treatment of such rings. A ring $A$ with a subset $E$ of commuting idempotent such that for any $x\in A$, there exists $e\in E$ such that $ex=xe=x$ is called a ring with local units. If $A$ is a graded ring and $E$ consists of homogeneous elements (of degree zero), we call $E$ a {\it set of homogeneous local units}.

\begin{proposition} \label{pkhti1}
Let $A$ be a $\Gamma$-graded ring with a set of homogeneous local units. The following statements are equivalent. 

\begin{enumerate}[\upshape(1)]
\item  $A$ is a graded von Neumann regular ring;

\item Any finitely generated right (left) graded ideal of $A$ is generated by one homogeneous
idempotent;
\end{enumerate}

If $A$ is a graded ring with unit, then the above statements are equivalent to 

\begin{enumerate}[\upshape(3)]

\item   Any graded right (left) R-module is flat. 
\end{enumerate}
\end{proposition}
\begin{proof}
The proof is similar to the non-graded case~\cite[Theorem~1.1, Corollary~1.13]{goodearlbook} and it is omitted. 
\end{proof}

We denote by $J^{\gr}(A)$ the graded Jacobson radical of $A$ and by $J(A)$ the usual Jacobson radical. 

\begin{proposition}\label{pkhti12}
Let $A$ be a (not necessarily unital) $\Gamma$-graded von Neumann regular ring. Then

\begin{enumerate}[\upshape(1)]
\item Any graded right (left) ideal of $A$ is idempotent;

\item Any graded ideal is semi-prime;

\item Any finitely generated right (left) graded ideal of $A$ is a projective module. 

\end{enumerate}

Furthermore, if $A$ is a $\mathbb Z$-graded regular ring with a set of homogeneous local units then, 

\begin{enumerate}[\upshape(4)]

\item $J(A)=J^{\gr}(A) =0$. 

\end{enumerate}
\end{proposition}
\begin{proof}
The proofs of (1)-(3)  are similar to the non-graded case~\cite[Corollary~1.2]{goodearlbook} and they are omitted. The statement (4) follows from Bergman's observation that $J(A)$ is a graded ideal (see~\cite[Lemma~6.2]{aarbitrary}). So by Proposition~\ref{pkhti1}, $J(A)$ contains an idempotent, which then forces $J(A)=0$. 
\end{proof}

In Proposition~\ref{cafejen4may} we will show that the graded von Neumann property is a graded Morita invariant. This will be used in the main Theorem~\ref{sthfin3}. 
For its proof, we need the following lemma which is the graded version of ~\cite[\S1, Lemma~1.6]{goodearlbook}. Its proof is similar  to the non-graded version, and we omit it.  
 
\begin{lemma} \label{monashhhh}
Let $A$ be a $\Gamma$-graded ring and $e_1,\dots, e_n$ be orthogonal homogeneous idempotent in $A$ such that $e_1+\dots+e_n=1$. Then $A$ is graded von Neumann regular if and only if for any $x\in e_i A^h e_j$, there exists $y\in e_j A^h e_i$ such that $xyx=x$.
\end{lemma}

By $\Gr A$ (or $\Gr[\Gamma] A$ to emphasis the graded group of $A$), we denote a category consists of $\Gamma$-graded right $A$-modules as objects and  graded homomorphisms as the morphisms. Recall that for $\alpha \in \Gamma$,  the $\alpha$-shift functor $\mathcal T_\alpha: \Gr A\rightarrow \Gr A$, $M \mapsto M(\alpha)$  is an isomorphism with the property
$\mathcal T_\alpha \mathcal T_\beta=\mathcal T_{\alpha + \beta}$, $\alpha,\beta\in \Gamma$. For two $\Gamma$-graded ring $A$ and $B$, 
a functor $\phi:\Gr A \rightarrow \Gr B$ is called a \emph{graded functor} \index{graded functor} if $\phi \mathcal T_{\alpha} = \mathcal T_{\alpha} \phi$. 
A graded functor $\phi:\Gr A \rightarrow \Gr B$ is called a \emph{graded equivalence} \index{graded equivalence} if there is a graded functor $\psi:\Gr B \rightarrow \Gr A$ such that $\psi \phi \cong 1_{\Gr A}$ and $\phi \psi \cong 1_{\Gr B}$.  If there is a graded equivalence between $\Gr A$ and $\Gr B$, we say $A$ and $B$ are \emph{graded equivalent} or \emph{graded Morita equivalence} and we write \index{graded Morita equivalent}
$\Gr A \approx_{\gr} \Gr B$, or $\Gr[\Gamma] A \approx_{\gr} \Gr[\Gamma] B$ to emphasis the categories are $\Gamma$-graded. We refer the reader to~\cite[\S2]{hazgrmon} for a comprehensive study of the graded Morita theory. In particular, to prove the next proposition, we need the following theorem (see ~\cite[Theorem~2.3.6]{hazgrmon}).

\begin{theorem}\label{pardosuggi}
Let $A$ and $B$ be $\Gamma$-graded rings. Then $A$ is graded Morita equivalent to $B$ if and only if \[B\cong_{\gr} e \M_n(A)(\overline \delta) e\] for a full homogeneous idempotent $e \in \M_n(A)(\overline \delta)$, where $\overline \delta=(\delta_1,\dots,\delta_n)$, $\delta_i \in \Gamma$.  
\end{theorem}

\begin{proposition}\label{cafejen4may}
Let $A$ and $B$ be $\Gamma$-graded rings which are graded Morita equivalent. Then $A$ is a graded von Neumann ring if and only if $B$ is a graded von Neumann ring.
\end{proposition}
\begin{proof}
Let $A$ be a graded Morita equivalent to $B$. Then by Theorem~\ref{pardosuggi}, 
\begin{equation}\label{waitingforj}
B\cong_{\gr} e \M_n(A)(\overline \delta) e,
\end{equation}
 for a full homogeneous idempotent $e \in \M_n(A)(\overline \delta)$, where $\overline \delta=(\delta_1,\dots,\delta_n)$, $\delta_i \in \Gamma$. Suppose $A$ is a graded regular ring. We will show that $B$ is graded regular. Since $e$ is a homogeneous idempotent, if we show that $\M_n(A)(\overline \delta)$ is graded regular then (\ref{waitingforj}) implies that $B$ is graded regular.  Let $e_i \in  \M_n(A)(\overline \delta)$ be a matrix with $1$ in $(i,i)$-component and zero elsewhere. 
 Clearly $e_i$, $1\leq i \leq n$, are orthogonal homogeneous idempotent (of degree zero) with $\sum_{i}e_i=1$. Since
 \begin{equation*}\label{mmkkhh}
\M_n(A)(\overline \delta)_{\la} =
\begin{pmatrix}
A_{ \la+\de_1 - \de_1} & A_{\la+\de_2  - \de_1} & \cdots &
A_{\la +\de_n - \de_1} \\
A_{\la + \de_1 - \de_2} & A_{\la + \de_2 - \de_2} & \cdots &
A_{\la+\de_n  - \de_2} \\
\vdots  & \vdots  & \ddots & \vdots  \\
A_{\la + \de_1 - \de_n} & A_{ \la + \de_2 - \de_n} & \cdots &
A_{\la + \de_n - \de_n}
\end{pmatrix},
\end{equation*} 
(see~\cite[\S1.1.3]{hazgrmon}), we have 
 \begin{equation}\label{didoyes}
 e_i \M_n(A)(\overline \delta)_\gamma e_j=e_{ij}(A_{\gamma+\delta_j-\delta_i}).
 \end{equation}
  Suppose  $x  \in e_i \M_n(A)(\overline \delta)_\gamma e_j$. By (\ref{didoyes}),  $x= e_{ij}(a)$, where $a \in A_{\gamma+\delta_j-\delta_i}$. 
 Since $A$ is graded regular, there is $b\in A_{-\gamma+\delta_i-\delta_j}$ such that $aba=a$.
 Thus $y:=e_{ji}(b) \in e_j \M_n(A)(\overline \delta)_{-\gamma} e_i$ and clearly $xyx=x$. Therefore by Lemma~\ref{monashhhh}, 
 $\M_n(A)(\overline \delta)$ is graded regular. Consequently, $B\cong_{\gr} e \M_n(A)(\overline \delta) e$ is graded regular. 
 
A similar argument shows that if $B$ is graded regular then $A$ is graded regular. This completes the proof. 
\end{proof}

\begin{remark}
One can also prove Proposition~\ref{cafejen4may} by observing that the graded Morita equivalent preserves the graded flatness along by using  Proposition~\ref{pkhti1}. 
\end{remark} 

\subsection{Leavitt path algebras}\label{hngt}

A {\it directed graph} $E=(E^0,E^1,r,s)$ consists of two sets $E^0$, $E^1$ and maps $r,s:E^1\rightarrow E^0$. The elements of $E^0$ are called {\it vertices} and the elements of $E^1$ {\it edges}. If $s^{-1}(v)$ is a finite set for every $v \in E^0$, then the graph is called {\it row-finite}.  In this setting, if the number of vertices, i.e.,  $|E^0|$,  is finite, then the number of edges, i.e.,  $|E^1|$, is finite as well and we call $E$ a {\it finite} graph. 

A vertex $v$ for which $s^{-1}(v)$ is empty is called a {\it sink}, while a vertex $w$ for which $r^{-1}(w)$ is empty is called a {\it source}. A vertex $v \in E^0$ such that $|s^{-1}(v)| = \infty$, is called an {\it infinite emitter}. If $v$ is either a sink or an infinite emitter, then it is called a {\it singular vertex}. If $v$ is not a singular vertex, it is called a {\it regular vertex}. In this note, we consider arbitrary graphs, i.e,   we don't assume the cardinality of vertices and edges are countable. We call a graph {\it countable} if the cardinality of vertices and edges are indeed countable. 

  An edge with the same source and range is called a {\it loop}. A path $\mu$ in a graph $E$ is a sequence of edges $\mu=\mu_1\dots\mu_k$, such that $r(\mu_i)=s(\mu_{i+1}), 1\leq i \leq k-1$. In this case, $s(\mu):=s(\mu_1)$ is the {\it source} of $\mu$, $r(\mu):=r(\mu_k)$ is the {\it range} of $\mu$, and $k$ is the {\it length} of $\mu$ which is  denoted by $|\mu|$. We consider a vertex $v\in E^0$ as a {\it trivial} path of length zero with $s(v)=r(v)=v$. By $E^n$, $n \in \mathbb N$, we denote the set of paths of length $n$.
If $\mu$ is a nontrivial path in $E$, and if $v=s(\mu)=r(\mu)$, then $\mu$ is called a {\it closed path based at} $v$. If $\mu=\mu_1\dots\mu_k$ is a closed path based at $v=s(\mu)$ and $s(\mu_i) \not = s(\mu_j)$ for every $i \not = j$, then $\mu$ is called a {\it cycle}.  

\begin{deff}\label{llkas}{\sc Leavitt path algebra of an arbitrary graph.} \label{LPA} \\
For a graph $E$ and a field $K$, the {\it Leavitt path algebra of $E$}, denoted by $\LL_K(E)$, is the algebra generated by the sets $\{v \mid v \in E^0\}$, $\{ \alpha \mid \alpha \in E^1 \}$ and $\{ \alpha^* \mid \alpha \in E^1 \}$ with the coefficients in $K$, subject to the relations 

\begin{enumerate}
\item $v_iv_j=\delta_{ij}v_i \textrm{ for every } v_i,v_j \in E^0$.

\smallskip

\item $s(\alpha)\alpha=\alpha r(\alpha)=\alpha \textrm{ and }
r(\alpha)\alpha^*=\alpha^*s(\alpha)=\alpha^*  \textrm{ for all } \alpha \in E^1$.

\smallskip
\item $\alpha^* \alpha'=\delta_{\alpha \alpha'}r(\alpha)$, for all $\alpha, \alpha' \in E^1$.

\smallskip
\item $\sum_{\{\alpha \in E^1, s( \alpha)=v\}} \alpha \alpha^*=v$, for every regular vertex $v\in E^0$.

\end{enumerate}
\end{deff}
Here the field $K$ commutes with the generators $\{v,\alpha, \alpha^* \mid v \in E^0,\alpha \in E^1\}$. Throughout this note the coefficient ring is a fixed field $K$ and we simply write $\LL(E)$ instead of $\LL_K(E)$. The elements $\alpha^*$ for $\alpha \in E^1$ are called {\it ghost edges}. One can show that $\LL(E)$ is a ring with identity if and only if the graph $E$ is finite (otherwise, $\LL(E)$ is a ring with local units). 

\subsection{Grading of Leavitt path algebras}\label{hngtrere}

For an arbitrary group $\Gamma$, one can equip $\LL(E)$ with a $\Gamma$-graded structure. This will be needed in the proof of the main Theorem~\ref{sthfin3}.  Let $\Gamma$ be an arbitrary group with the identity element $e$. Let $w:E^1\rightarrow \Gamma$ be a {\it weight} map and further define $w(\alpha^*)=w(\alpha)^{-1}$, for $\alpha \in E^1$ and $w(v)=e$ for $v\in E^0$.  The free $K$-algebra generated by the vertices, edges and ghost edges is a $\Gamma$-graded $K$-algebra. Furthermore, the Leavitt path algebra is the quotient of this algebra by relations in  Definition~\ref{LPA} which are all homogeneous. Thus $\LL_K(E)$ is a $\Gamma$-graded $K$-algebra. 

The {\it canonical} grading given to a Leavitt path algebra is a $\mathbb Z$-grading by setting $\deg(v)=0$, for $v\in E^0$, $\deg(\alpha)=1$ and $\deg(\alpha^*)=-1$ for $\alpha \in E^1$. 
If $\mu=\mu_1\dots\mu_k$, where $\mu_i \in E^1$, is an element of $\LL(E)$, then we denote by $\mu^*$ the element $\mu_k ^*\dots \mu_1^* \in \LL(E)$. Further we define $v^*=v$ for any $v\in E^0$. Since $\alpha^* \alpha'=\delta_{\alpha \alpha'}r(\alpha)$, for all $\alpha, \alpha' \in E^1$, any word in the generators $\{v, \alpha, \alpha^* \mid v\in E^0, \alpha \in E^1   \}$ in $\LL(E)$ can be written as $\mu \gamma ^*$ where $\mu$ and $\gamma$ are paths in $E$ (vertices are considered paths of length zero).  The elements of the form $\mu\gamma^*$ are called {\it monomials}. 

Taking the grading into account, one can write $\LL(E) =\textstyle{\bigoplus_{k \in \mathbb Z}} \LL(E)_k$ where,
\[\LL(E)_k=  \Big \{ \sum_i r_i \alpha_i \beta_i^*\mid \alpha_i,\beta_i \textrm{ are paths}, r_i \in K, \textrm{ and } |\alpha_i|-|\beta_i|=k \textrm{ for all } i \Big\}.\]
\subsection{Corner skew Laurent polynomial rings}\label{cornerskew}

The Leavitt path algebras of finite graphs with no sources are examples of corner skew Laurent polynomial rings~\cite{arafrac}. We will use this interpretation to prove that the Leavitt path algebras of such finite graphs are graded regular. We recall the construction of these rings here. 

Let $R$ be a ring with identity and $p$ an idempotent of $R$. Let $\phi:R\rightarrow pRp$ be a {\it corner} isomorphism, i.e, a ring isomorphism with $\phi(1)=p$. A \emph{corner skew Laurent polynomial ring}  \index{corner skew Laurent polynomial ring} with coefficients in $R$, denoted by $R[t_{+},t_{-},\phi]$, is a unital ring which is constructed as follows:  The elements of $R[t_{+},t_{-},\phi]$ are formal expressions
\[t^j_{-}r_{-j} +t^{j-1}_{-}r_{-j+1}+\dots+t_{-}r_{-1}+r_0 +r_1t_{+}+\dots +r_it^i_{+},\]
where $r_{-n} \in p_n R$ and $r_n \in R p_n$, for all $n\geq 0$, where $p_0 =1$ and $p_n =\phi^n(p_0)$. The addition is component-wise, and multiplication is determined by the distribution law and the following rules:
\begin{equation}\label{oiy53} 
t_{-}t_{+} =1, \qquad t_{+}t_{-} =p, \qquad rt_{-} =t_{-}\phi(r),\qquad  t_{+}r=\phi(r)t_{+}.
\end{equation}

The corner skew Laurent polynomial rings is a special case of a so called fractional skew monoid rings constructed in~\cite{arafrac}. Assigning $-1$ to $t_{-}$ and $1$ to $t_{+}$ makes $A:=R[t_{+},t_{-},\phi]$ a $\mathbb Z$-graded ring with $A=\bigoplus_{i\in \mathbb Z}A_i$, where (see~\cite[Proposition~1.6]{arafrac})
\begin{align*}
A_i& = Rp_it^i_{+}, \text{ for  } i>0,\\
A_i&=t^i_{-}p_{-i}R, \text{ for } i<0,\\
A_0& =R. 
\end{align*}
Clearly, when $p=1$ and $\phi$ is the identity map, then $R[t_{+},t_{-},\phi]$ reduces to the familiar ring $R[t,t^{-1}]$. 


\begin{proposition}\label{lanhc8}
Let $R$ be a ring with identity and $A=R[t_{+},t_{-},\phi]$ be a corner skew Laurent polynomial ring. Then $A$ is a graded von Neumann regular ring if and only if $R$ is a von Neumann regular ring. 
\end{proposition}
\begin{proof}
If a graded ring is graded von Neumann regular, then it is easy to see that its zero component ring is von Neumann regular. This proves one direction of the theorem. For the converse, suppose $R$ is regular.  Let $x \in A_i$, where $i>0$. So $x=rp_it_{+}^i$, for some $r\in R$, where $p_i=\phi^i(1)$.  By relations (\ref{oiy53}) and induction, we have $t_{+}^i t_{-}^i=\phi^i(p_0)=p_i$. Since $R$ is regular, there is an $s\in R$ such that $rp_i s rp_i =rp_i$. Then choosing 
$y=t^i_{-}p_is$, we have 
\[xyx=(rp_it_{+}^i)(t^i_{-}p_is)(rp_it_{+}^i)=(rp_it_{+}^i t^i_{-}p_is)(rp_it_{+}^i)=  rp_i p_i p_i s rp_it_{+}^i= rp_i t_{+}^i=x.\]
A similar argument shows that for $x \in A_i$, where $i<0$, there is a $y$ such that $xyx=x$. This shows that $A$ is a graded von Neumann regular ring. 
\end{proof}

\subsection{} \label{calgaryplaza}
We can realise the Leavitt path algebras of finite graphs with no sources in terms of corner skew Laurent polynomial rings~\cite{arafrac}. 
Let $E$ be a finite graph with no sources and let $E^0=\{v_1,\dots,v_n\}$ be the set of all vertices of $E$. For each $1\leq i \leq n$, we choose an edge $e_i$ such that $r(e_i)=v_i$ and consider $t_{+}=e_1+\dots+e_n \in \LL(E)_1$. Then $t_{-}= e^*_1+\dots+e^*_n$ is its left inverse.  
Thus by~\cite[Lemma~2.4]{arafrac}, \[\LL(E)=\LL(E)_0[t_{+},t_{-},\phi],\] where 
\begin{align*}
\phi:\LL(E)_0&\longrightarrow t_{+}t_{-}\LL(E)_0 t_{+}t_{-},\\
a &\longmapsto t_{+}at_{-}.
\end{align*} 

\subsection{Desingularization}\label{hngtrere66}

We briefly recall the desingularization of a countable graph from~\cite{aarbitrary,drinentom}. This will be used in the proof of the main Theorem~\ref{sthfin3}, to extend the graded regularity from row-finite graphs to countable graphs. 

If $E$ is a directed countable graph, then a {\it desingularization} of $E$ is a graph $F$ formed by adding a tail to every sink and every infinite emitter of $E$ in the following manner: If $v_0$ is a sink in $E$, then by adding a tail at $v_0$ we mean attaching a graph of the form
\[\xymatrix{
v_0 \ar[r] & v_1 \ar[r] & v_2 \ar[r]  & v_3 \ar@[.>][r] & \cdots 
}
\]
to $E$ at $v_0$. If $v_0$ is an infinite emitter in $E$, then by adding a tail at $v_0$ we mean performing the following process. We first list the edges $e_1, e_2, e_3, \dots$ of $s^{-1}(v_0)$. Then we add a tail to $E$ at $v_0$ of the following form
\[\xymatrix{
v_0 \ar[r]^{f_1} & v_1 \ar[r]^{f_2} & v_2 \ar[r]^{f_3}  & v_3 \ar@[.>][r] & \cdots 
}
\]

We remove the edges in $s^{-1}(v_0)$, and for every $e_j \in s^{-1}(v_0)$ we draw an edge $g_j$ from $v_{j-1}$ to $r(e_j)$. 

For example, if $E$ is a graph with one vertex with infinitely many loops, then its desingularization is 

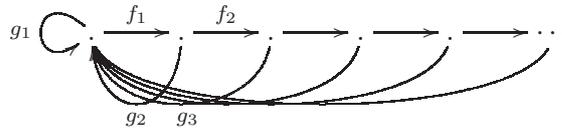
\[
\xymatrix{
. \ar[r]^{f_1} \ar@(ul,dl)[]_{g_1} & . \ar[r]^{f_2} \ar@(d,d)[l]^{g_2} & . \ar[r] \ar@(d,d)[ll]^{\;\;g_3} &
. \ar[r] 
\ar@(d,d)[lll] & . \ar[r] \ar@(d,d)[llll] & \ar@(d,d)[lllll] \cdots\\
&&&&&\\
}
\]

\subsection{Main Theorem}

We are in a position to prove the main theorem of the note.

\begin{theorem}\label{sthfin3}
Let $E$ be an arbitrary graph.  Then $\LL(E)$ is a graded von Neumann regular ring. 
\end{theorem}
\begin{proof}
We prove the theorem first for finite graphs with a finite number of edges. Since a row-finite graph is a direct limit of finite graphs, the theorem follows for Leavitt path algebras associated to row-finite graphs from the finite case. We then use this along with the desingularization to extend the theorem for arbitrary graphs.

\noindent {\bf I. Finite graphs with a finite number of edges.}
 
Let $E$ be a finite graph. First note that if a graph $E$ has an isolated vertex $v$, then $E_{\backslash v}$ is the graph obtained by removing $v$ from $E$. Then  we have 
$\LL(E)\cong_{\gr} K \oplus \LL(E_{\backslash v})$, where $K$ is a graded ring concentrated in degree $0$. It is now easy to see that $\LL(E)$ is graded regular if and only if $\LL(E_{\backslash v})$ is graded regular. 

Next, suppose that $E$ is a graph with at least two vertices. Let $v \in E^0$ be a source which is not a sink (i.e, it is not an isolated vertex). Then we show that 
$\LL(E)$ is graded Morita equivalent to $\LL(E_{\backslash v})$. 

Since $E_{\backslash v}$ is a complete subgraph of $E$, there is a  (non-unital) graded algebra homomorphism $\phi:\LL(E_{\backslash v}) \rightarrow \LL(E)$, such that $\phi(u)=u$, $\phi(e)=e$ and $\phi(e^*)=e^*$, where $u \in E^0 \backslash \{v\}$ and $e\in E^1 \backslash  \{f \in E^1\mid s(f)=v\}$.  
The graded uniqueness theorem~\cite[Theorem~4.8]{tomforde} implies  $\phi$ is injective. Thus  $\LL(E_{\backslash v})  \cong_{\gr} 
\phi(\LL(E_{\backslash v}))$. It is not difficult to see that $\phi(\LL(E_{\backslash v})) =p\LL(E) p$, where $p=\sum_{u \in E_{\backslash v}^0}u$. 
This immediately implies that $\LL(E_{\backslash v})$ is  graded Morita equivalent to $p\LL(E)p$. On the other hand, the (graded) ideal generated by $p=\sum_{u \in E_{\backslash v}^0}u$ coincides with the (graded) ideal generated by $\{u \mid u \in E_{\backslash v}^0 \}$. We show that the graded ideal $I$ generated by $\{u \mid u \in E_{\backslash v}^0 \}$ contains $v$. For \[v=\sum_{\{f \in E^1\mid s(f)=v\}}f r(f) f^*,\] where the sum is over a nonempty set (as $v$ is not an isolated vertex). But $r(f) \in I$. Thus $v \in I$ and so $1=\sum_{u \in E^0}u \in I$. That is $I=\LL(E)p\LL(E)=\LL(E)$.  
This shows that $p$ is a full homogeneous idempotent in $\LL(E)$. 
Thus $p\LL(E)p$ is graded Morita equivalent to $\LL(E)$ (see~\cite[\S 2, Example~2.3.1]{hazgrmon}). Consequently, $\LL(E_{\backslash v})$
is graded Morita equivalent to $\LL(E)$. 
Since by Proposition~\ref{cafejen4may} the graded regular property is graded Morita invariant, it follows 
that $\LL(E)$ is graded regular if and only if $\LL(E_{\backslash v})$ is graded regular. 

 Now let $E'$ be a graph with at least one vertex obtained from $E$ by repeatedly removing all the sources (and thus isolated vertices). An easy induction shows that $E'$ is either an isolated vertex or a graph with no sources.  The preceding argument now shows that $\LL(E)$ is graded regular if and only if $\LL(E')$ is graded regular. 

To finish the proof, we consider two cases. 
If $E'$ is an isolated vertex, then $\LL(E')\cong_{\gr} K$ and so it is a graded regular ring. Thus $\LL(E)$ is graded regular.

If $E'$ is a graph with no sources, then $\LL(E')$ can be realised as a corner skew  Laurent polynomial ring  $\LL(E')=\LL(E')_0[t_{+},t_{-},\phi]$ as in \S\ref{calgaryplaza}.  But $\LL(E')_0$, being an ultramatricial algebra, is a regular ring (see the proof of~\cite[Theorem~5.3]{amp}). Thus by Proposition~\ref{lanhc8}, $\LL(E')$ is graded regular and  consequently $\LL(E)$ is graded regular. 

\noindent {\bf II. Row-finite graphs.}

 Let $E$ be a row-finite graph. Since any row-finite graph is a direct limit of finite graphs (these are complete subgraphs with complete graph homomorphisms), it follows that $\LL(E)$ is a direct limit of (graded) Leavitt path algebras associated  to finite graphs (see~\cite[Lemmas~3.1, 3.2]{amp}). Since a direct limit of graded regular rings is a graded regular, from (I), it follows that $\LL(E)$ is graded regular. 

\noindent {\bf III. Countable graphs.}

Let $E$ be a countable arbitrary graph which has infinite emitters. Let $F$ be a fixed desingularization graph $E$ (see~\S\ref{hngtrere66}). 
We assign a (non-canonical) grading to $\LL(E)$ as follows (see~\S\ref{hngtrere}). For $v\in E^0$, set $\deg(v)=0$. For $e\in E^1$, where $s(e)$ is regular, set $\deg(e)=1$ and $\deg(e^*)=-1$.  Let $e\in E^1$, where $s(e)$ is an infinite emitter. In the process of obtaining the desingularization graph $F$, we would have named 
$e$ as $e_i$ for some $i \geq 1$, so that the ``substitute'' for the edge $e = e_i$ of $E$ is the path $f_1\dots f_{i-1}g_i$ in $F$. We then define $\deg(e)=i$ and $\deg(e^*)=-i$. This induces a $\mathbb Z$-grading on $\LL(E)$. Consider the canonical grading of $\LL(F)$. By~\cite[Proposition~5.1]{aarbitrary}, the assignment  
\begin{align*}
\phi: \LL(E) & \longrightarrow \LL(F),\\
v\phantom{^*} &\longmapsto v,\\
e\phantom{^*} &\longmapsto e, \qquad & s(e) \text { is regular, }\\
e^* &\longmapsto e^*, \qquad & s(e) \text { is regular, }\\
e_i &\longmapsto f_1\dots f_{i-1}g_i, \qquad & s(e_i) \text { is infinite emitter, }\\
e^*_i &\longmapsto g_i^* f_{i-1}^*\dots f_{1^*}, \qquad & s(e_i) \text { is infinite emitter, }\\
\end{align*}
induces an monomorphism. Note that with the grading assigned to $\LL(E)$ and $\LL(F)$ above, $\phi$ is a graded morphism. List the vertices of $E$ as 
$E^0=\{v_i\}_{i\in \mathbb N}$ if $|E^0|=\infty$ and $E^0=\{v_1,\dots,v_n\}$ otherwise. 
Notice that, 
\[\nu_n \LL(F) \nu_n \subseteq \nu_{n+1} \LL(F) \nu_{n+1},\] and 
\begin{equation}\label{omjhiu}
\phi(\LL(E)) =\bigcup \nu_n \LL(F) \nu_n,
\end{equation}
where $\nu_n=\sum_{1\leq i\leq n}v_n$ (here, $v_i\in \LL(F)$ is the image of the corresponding vertex of $E$ under $\phi$).  
Since $F$ is a row-finite graph, by (II), $\LL(F)$ is a graded regular ring. Now since each $\nu_n$ is an idempotent, $\nu_n \LL(F) \nu_n$ is graded regular, and thus $\bigcup \nu_n \LL(F) \nu_n$ is graded regular. Since $\phi$ is a graded monomorphism, (\ref{omjhiu}) implies that $\LL(E)$ is graded regular with the given non-canonical grading defined above. Finally suppose that a ring $R$ has two different gradings, say a $\Gamma$-grading and $\Omega$-grading such that their homogeneous elements coincides. Then it is easy to see that $\Gamma$-graded ring $R$ is graded regular if and only if $\Omega$-graded ring $R$ is graded regular. Using this, it follows that $\LL(E)$ is graded regular with the canonical grading of Leavitt path algebras as well.

\noindent {\bf IV. Arbitrary graphs.}

Let $E$ be an uncountable graph. Then by~\cite[Proposition~2.7]{goodearllpa}, $E$ is the direct limit of its countable  CK-subgraphs (see~\cite[\S2.3]{goodearllpa}), and consequently $\LL(E)$ is the direct limit of the (graded) $\LL(F)$ over countable CK-subgraphs $F$ of $E$. By (III), each of $\LL(F)$ is graded regular. Thus $\LL(E)$, being their direct limit, is graded regular. 
This completes the proof. 
  \end{proof}

\subsection{Conclusion}

A Leavitt path algebra is a $\mathbb Z$-graded ring with a set of homogeneous local units (the set of finite sums of distinct vertices) and further if the number of vertices of the graph is finite, then the algebra is unital. Thus Theorem~\ref{sthfin3} gives that Leavitt path algebras are  graded von Neumann regular ring with a set of homogeneous local units. Thus all the statements of Propositions~\ref{pkhti1},~\ref{pkhti12} hold for Leavitt path algebras accordingly.

Finally, Leavitt path algebras are not (graded) unit-regular. As an example, consider the following graph:
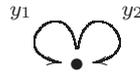
\[
\xymatrix{
 \bullet\ar@(u,l)_{y_1} \ar@(u,r)^{y_2}}
\]
Then it is easy to see that there is no homogeneous invertible element $x$ such that $y_1 x y_1=y_1$.

\end{document}